\documentclass[12pt]{amsart}

\usepackage{amssymb}

\usepackage[all]{xy}

\makeatletter
\@namedef{subjclassname@2020}{%
  \textup{2020} Mathematics Subject Classification}
\makeatother

\usepackage[T1]{fontenc}

\newtheorem{theorem}{Theorem}[section]

\newtheorem{lemma}[theorem]{Lemma}
\newtheorem{proposition}[theorem]{Proposition}

\theoremstyle{definition}

\newtheorem{remark}[theorem]{Remark}
\newtheorem{example}[theorem]{Example}

\numberwithin{equation}{section}

\DeclareMathOperator{\Aa}{\mathcal{A}}

\newcommand{\lra}{\longrightarrow}

\newcommand{\oQ}{\overline{Q}}

\newcommand{\benu}{\begin{enumerate}}
\newcommand{\enu}{\end{enumerate}}

\newcommand{\A}{\mathbb{A}}
\newcommand{\D}{\mathbb{D}}
\newcommand{\E}{\mathbb{E}}

\begin{document}

\baselineskip=17pt


\title[Representation type]{A note on the representation type of generalized path algebras}

\author[Chust]{Viktor Chust}
\address{(Viktor Chust) Institute of Mathematics and Statistics - University of São Paulo, São Paulo, Brazil}
\email{viktorchust.math@gmail.com}

\author[Coelho]{Fl\'avio U. Coelho}
\address{(Flávio U. Coelho) Institute of Mathematics and Statistics - University of São Paulo, São Paulo, Brazil}
\email{fucoelho@ime.usp.br}

\date{}



\maketitle

\begin{abstract}
  In \cite{ICNLP}, the authors describe the ordinary quiver of a given generalized path algebra, a concept introduced by Coelho and Liu in \cite{CLiu}. In this short note, we use this result to characterize which generalized path algebras have finite or tame representation type. 
\end{abstract}

\section{Introduction}

One of the questions involving the study of algebra, under the representation theory perspective, concerns its representation type,  an invariant which measures, intuitively speaking, the complexity of the category of modules. We say, for instance, that the algebra is {\it representation-finite} if, clearly, there are only finitely many non-isomorphic indecomposable modules. The algebras which are not representation-finite can be divided initially in two classes, the so called tame and wild algebras. In this note, we will be concerned in analyising the representation type of a special class of algebras which we shall now describe.

Let $k$ be an algebraically closed field. Here, an algebra $A$ will mean a basic finite dimensional $k$-algebra and we shall also assume it is indecomposable as an algebra. In this setting, a well-known result due to Gabriel states that there exists a (connected) quiver $Q$ such that $A$ is isomorphic to the quotient of the path algebra $kQ$ by an admissible ideal (see below for details). Theorems  \ref{dynkin} and \ref{euclidean} below describe when a given path algebra $kQ$ is representation-finite or of tame type in terms of $Q$.

The notion of path algebras can be generalized by using a quiver $Q$ and assigning to each vertex $i$ of $Q$ an algebra $A_i$ instead of just $k$ as in the traditional path algebra construction (see \cite{CLiu}). They are called generalized path algebras. 

The purpose of this short note is to give a corresponding description of finite and tame type in this generalized setting. We shall recall the necessary concepts in Section \ref{prel}, while Sections \ref{finiteness} and \ref{tameness}  will be devoted to the announced discussion. 

A last remark concerning the general case, that is, quotients of generalized path algebras. It is still a very complex task to describe all representation-finite algebras through the quiver and the relations which generate the admissible ideal. When we restrict our study to the generalized path algebras and not to a possible quotient of it we do not have to deal with the relations on its defining quiver. 

\section{Preliminaries} \label{prel}

We shall here recall some notions needed in the sequel. For details on the basics of the representation theory, we refer the reader to \cite{AC,ASS} and on generalized path algebras to \cite{CC1,CLiu}.

\subsection{Path algebras} A \textbf{quiver} is a tuple $Q = (Q_0,Q_1,s,e)$, where $Q_0$ and $Q_1$ are sets and $s,e: Q_1 \rightarrow Q_0$ are functions. The elements of $Q_0$ are called \textbf{vertices} while the elements of $Q_1$ are called \textbf{arrows}, and given an arrow $\alpha \in Q_1$, $s(\alpha)$ is called the \textbf{start vertex} of $\alpha$, and $e(\alpha)$ is called the \textbf{target vertex} of $\alpha$. An arrow starting and ending at the same vertex will be called a {\bf loop}. An {\bf end point of a quiver $Q$} is a vertex which is involved in at most one arrow. The {\bf underlying graph $\overline{Q}$ of a quiver} $Q$ is given by the same set of vertices and with {\bf edges} replacing the arrows. We shall assume here that all quivers are finite, that is, the sets $Q_0$ and $Q_1$ are finite. 

A {\bf path} in a quiver $Q$ is given by a concatenation of arrows. The path algebra $kQ$ of the quiver $Q$ is defined as follows. The set of all paths of $Q$ is a basis of $kQ$ as a $k$-vector space. The natural concatenation of paths gives a multiplication in $kQ$ which, spanned to all its elements, will give an algebra structure. Clearly, $kQ$ will be finite dimensional if and only if there are no oriented cycles in $Q$, that is, no paths starting and ending at the same vertex. Such an algebra will be referred as the {\bf path algebra} of $Q$.

An ideal $I$ of a path algebra is called {\bf admissible} provided: (i) $I$ is generated by combinations of paths of length at least two; and (ii) there exists an $m$ such that all paths of length at least $m$ lie inside $I$. We recall the following well-established fact (see \cite{AC} for instance). 

\begin{theorem}  \label{gabriel}
A finite dimensional basic algebra over an algebraically closed field $A$ is isomorphic to a quotient of a path algebra $kQ_A$ by an admissible ideal. 
\end{theorem}

In this case, the quiver $Q_A$ is called the {\bf ordinary quiver of $A$}.

\subsection{Generalized path algebras} \label{sub:gpalgebras}
As mentioned in the introduction, we are interested in a generalization of the path algebras. This generalization was introduced in \cite{CLiu} and goes as follows. Let $\Gamma$ denote a quiver without oriented cycles and $\Aa = \{ A_i \colon i \in \Gamma_0\}$ denote a family of basic finite dimensional $k$-algebras. To such data, we assigned  an algebra $\Lambda = k(\Gamma, \Aa)$ with a natural multiplication given not only by the concatenation of paths of the quiver but also by those of the algebras $A_i$ associated with the vertices of $\Gamma$. Such an algebra is called a {\bf generalized path algebra}, of {\bf gp-algebra}, for short. We shall keep the above notation whenever we talk on a gp-algebra $\Lambda = k(\Gamma, \Aa)$. See \cite{CC1} for further details.

Obviously, any algebra $A$ can be seen as a gp-algebra in a trivial way. Just consider the quiver $\Gamma$ with a unique vertex and no arrows and $\Aa = \{ A \}$. So, $k (\Gamma, \Aa) \cong A$.

\subsection{Realizing a gp-algebra as a path algebra} Let $\Lambda =  k(\Gamma, \Aa)$ be a gp-algebra. For each $i \in \Gamma_0$, due to Theorem \ref{gabriel}, $A_i \cong \frac{k \Sigma_i}{I_i}$, where  $\Sigma_i$ a quiver, and $I_i$ is an admissible ideal of $k\Sigma_i$. Also, by the same result,  $\Lambda$ is isomorphic to a quotient of a path algebra $kQ$ by an admissible ideal $J$. 

One could wonder how to construct the ordinary quiver $Q$ of $\Lambda$ from the quiver $\Gamma$ and the ordinary quivers $\Sigma_i$ of the algebras $A_i$. This has been done in  \cite{ICNLP} (see also \cite{CC1} for further details and generalizations), but for convenience, we shall repeat it here.

Keeping the above notations, the quiver $Q$ is constructed as follows. Its set of vertices $Q_0$ is the union $\bigcup_{i\in \Gamma_0} (\Sigma_i)_0$ of the vertices of the quivers $\Sigma_i, i \in \Gamma_0$. For the arrows in $Q_1$, they are of two kinds. Let $a, b \in Q_0$. So, there are $i, j$ such that $a \in (\Sigma_i)_0$ and $b \in (\Sigma_j)_0$.
\benu
\item[(i)] If $i \neq j$, then the number of arrows $a \rightarrow b$ in $Q_1$ equals the number of arrows $i \rightarrow j$ in $\Gamma_1$. These arrows are called {\bf arrows of type I}.
\item[(ii)] If $ i = j$, then the number of arrows $a \rightarrow b$ in $Q_1$ equals the number of arrows $a \rightarrow b$ in $(\Sigma_i)_1$. These arrows are called {\bf arrows of type II}.
\enu

Accordingly, we will say that an edge in the underlying graph of $Q$ is of type I or II if the corresponding arrow is of the respective type. 

\begin{remark} \label{typeI}
Let $kQ/I$ be the quotient of a path algebra $kQ$ by an admissible ideal. Observe that if the underlying graph $\overline{Q}$  contains at most one edge of type II, then there are no relations of $I$ involving the corresponding arrows. 
\end{remark}

Observe finally that the ordinary quiver $Q$ of $\Lambda$ contains a 
subquiver equal to $\Gamma$. We emphasize that $\Gamma$ has no oriented cycles (because we are dealing with finite dimensional algebras). We shall keep the above notations for the rest of the work. 

\begin{example}
\label{ex:gen_icnlp}
Let $\Lambda$ be the gp-algebra given by the quiver

{\small \begin{displaymath}
\xymatrix{
{\begin{matrix}
 k \\ \bullet \\ 1
 \end{matrix}}
\ar[rr]^{\alpha}="a" && 
{\begin{matrix}
 \frac{k\Sigma_2}{I_2} \\ \bullet \\ 2
 \end{matrix}}
\ar@<0.5ex>[rr]^{\beta}="b" \ar@<-0.5ex>[rr]_{\gamma} &&
{\begin{matrix}
 k \\ \bullet \\ 3
 \end{matrix}} }
\end{displaymath} }
where $\Sigma_2$ is the quiver
\begin{displaymath}
\xymatrix{
\bullet \ar[d]^{\delta}\\
\bullet \ar[d]^{\varepsilon}\\ 
\bullet}
\end{displaymath}
and $I_2 = (\delta \varepsilon)$. The ordinary quiver of $\Lambda$ is given by

{\small 
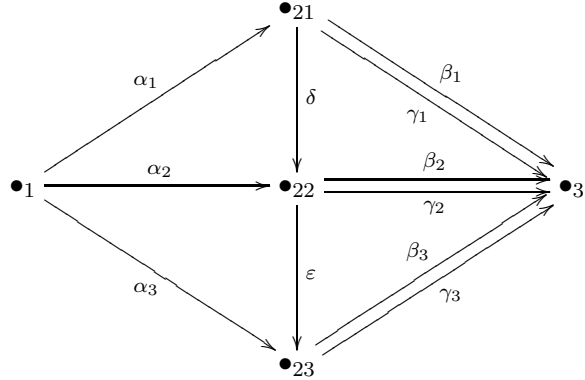
\begin{displaymath}
\xymatrix{
&&& \bullet_{21} \ar[dd]^{\delta} \ar@<0.5ex>[ddrrr]^{\beta_1} \ar@<-0.5ex>[ddrrr]_{\gamma_1}&&& \\
&&&&&&\\
\bullet_1 \ar[uurrr]^{\alpha_1} \ar[rrr]^{\alpha_2} \ar[ddrrr]_{\alpha_3} &&& \bullet_{22} \ar[dd]^{\varepsilon} \ar@<0.5ex>[rrr]^{\beta_2} \ar@<-0.5ex>[rrr]_{\gamma_2} &&& \bullet_3 \\
&&&&&&\\
&&& \bullet_{23} \ar@<0.5ex>[uurrr]^{\beta_3} \ar@<-0.5ex>[uurrr]_{\gamma_3} &&& }
\end{displaymath} }
and $J = (\delta \varepsilon)$.
\end{example}

\subsection{Representations} A {\bf bound quiver} is given by a pair $(Q,I)$, where $Q$ is a quiver while $I$ is a set of relations of the paths in $Q$ which generates an admissible ideal of $kQ$. Denote by rep$_k(Q,I)$ the category of finite dimensional $k$-representations of $Q$ bound by the relations $I$ (see \cite{AC}). It is a well-established fact that there is  a $k$-linear equivalence  between rep$_k(Q,I)$ and the category mod$kQ/I$ of finitely generated $kQ/I-$modules. 

A similar equivalence can be thought in the generalized case. If $\Lambda = k(\Gamma, \Aa)$, then the categories rep$_k(\Gamma, \Aa)$ of the finite dimensional $k$-representations of $(\Gamma, \Aa)$ and the category of $k(\Gamma, \Aa)-$modules are equivalent. Recall that a finite dimensional $k$-representation of $(\Gamma, \Aa)$ is given by an $A_i$-module $M_i$ at each vertex $i\in \Gamma_0$ and usual $k$-linear transformations assigned to each arrow. Recall the following result from \cite{ICNLP}.

\begin{proposition} \label{prop:rep}
Let $\Lambda = k(\Gamma, \Aa)$ be a gp-algebra and $\Lambda \cong kQ/I$ its realization as a bound path algebra, then $\Gamma$ is a subquiver of $Q$ and there exists a full and faithful functor rep$_k(\Gamma) \lra $ rep$_k (Q,I)$. 
\end{proposition}

\section{Representation-finiteness} \label{finiteness}

From now on, $\Lambda = k(\Gamma, \Aa) \cong kQ/I$ will denote a gp-algebra and its realization as a bound quiver algebra. 
In this section, we shall analise the case where the algebras are representation-finite, that is, algebras admitting only finitely many nonisomorphic  indecomposable modules. 
The case of path algebras is well-established by the theorem below (see \cite{ASS}). 
\vspace{.3 cm}

\begin{theorem} \label{dynkin}
Let $A = kQ$ be a path algebra given by the quiver $Q$.
Then $A$ is representation-finite if and only if the underlying graph of $Q$ is one of the following:
{\small \begin{displaymath}
\xymatrix{\mathbb{A}_n: & 1 \ar@{-}[r] & 2 \ar@{-}[r] & 3 \ar@{-}[r] & \cdots \ar@{-}[r] & n & (n\geq 1)}
\end{displaymath}
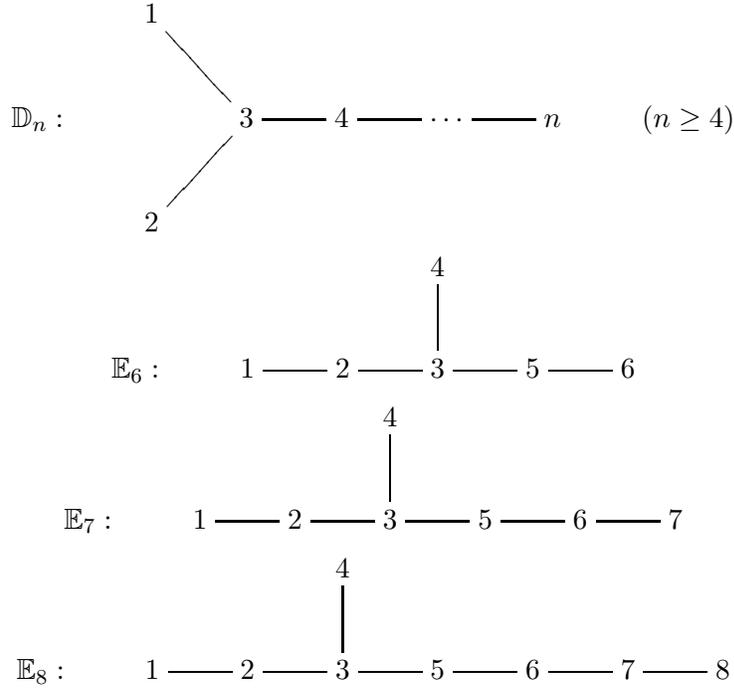
\begin{displaymath}
\xymatrix{& 1\ar@{-}[dr]&&&&& \\
\mathbb{D}_n: &  & 3 \ar@{-}[r] & 4 \ar@{-}[r] & \cdots \ar@{-}[r] & n & (n\geq 4)\\
& 2 \ar@{-}[ur]&&&&&}
\end{displaymath}
\begin{displaymath}
\xymatrix{& & &4 \ar@{-}[d]&& \\
\mathbb{E}_6: & 1 \ar@{-}[r]& 2 \ar@{-}[r] & 3 \ar@{-}[r] & 5 \ar@{-}[r] & 6}
\end{displaymath}
\begin{displaymath}
\xymatrix{& & &4 \ar@{-}[d]&&& \\
\mathbb{E}_7: & 1 \ar@{-}[r]& 2 \ar@{-}[r] & 3 \ar@{-}[r] & 5 \ar@{-}[r] & 6 \ar@{-}[r]& 7}
\end{displaymath}
\begin{displaymath}
\xymatrix{& & &4 \ar@{-}[d]&& &&\\
\mathbb{E}_8: & 1 \ar@{-}[r]& 2 \ar@{-}[r] & 3 \ar@{-}[r] & 5 \ar@{-}[r] & 6 \ar@{-}[r] & 7 \ar@{-}[r]& 8}
\end{displaymath} }
\end{theorem}

The graphs listed in the above theorem are called {\bf Dynkin} and we shall keep in the sequence the notation and the numbering established there. 

If $\Gamma$ is of type $\A_1$ and $\Aa = \{ A_1 \}$, then by the remark made in subsection  \ref{sub:gpalgebras}, $\Lambda \cong A_1$ and so, $\Lambda$ is representation-finite if and only if $A_1$ is. Next result deals with $\Gamma$  different from $\A_1$. 
\vspace{.3 cm} 

\begin{theorem} \label{repfinite}
Let $\Lambda = k(\Gamma, \Aa) \cong kQ/I$ be a gp-algebra with $|\Gamma_0| \geq 2$ and its realization as a quotient of a path algebra. Assume that none of the ordinary quivers $\Sigma_i$ of the algebras $A_i$ in $\Aa$ have loops. Then $\Lambda$ is representation-finite if and only if the underlying graph of $\Gamma$ is Dynkin and 
\benu
\item[(i)] If $\Gamma$ is of type $\D_n$ ($n \geq 4$),  $\E_6, \E_7$ or $ \E_8$, then $A_i = k$ for each $i \in \Gamma_0$.
\item[(ii)] If $\Gamma$ is of type $\A_2$, then either 
\begin{enumerate}
\item $A_1 = k= A_2$; or
\item one and only one of $A_1$ or $A_2$ is $k^l$ with $l = 2, 3$ and the other is $k$.
\end{enumerate} 
\item[(iii)] If  $\Gamma$ is of type $\A_n$ with $n \geq 3$, then either \begin{enumerate}
\item $A_i = k$ for each $i \in \Gamma_0$; or 
\item one and only one of $A_1$ or $A_n$ is $k^2$ and all the others algebras are $k$.
\end{enumerate} 
\enu
Moreover, in this case, $I = 0$ and $\Lambda$ is a path algebra of Dynkin type. 
\end{theorem}

Observe that if (i) holds, then $\Lambda$ is indeed a path algebra. In order to prove the theorem, we shall analise how the dimensions of the algebras $A_i \in \Aa$ ($i \in \Gamma_0$) afects the construction of the quiver $Q$. For convenience, we start with a lemma. 
\vspace{.3 cm}

\begin{lemma}  \label{lemrepfinite}
Let $\Lambda = k(\Gamma, \Aa) \cong kQ/I$ be a representation-finite gp-algebra with $|\Gamma_0| \geq 2$ and its realization as a quotient of a path algebra. Assume that none of the ordinary quivers $\Sigma_i$ of the algebras $A_i$ in $\Aa$ have loops and let $i \in \Gamma_0$. If dim$_kA_i \geq 2$,  then $i$ is an end point of $\Gamma$  and $\Gamma$ is of type $\A_n$ for some $n \geq 2$. Moreover, only one of the two end points $i$ of $\A_n$ satisfies dim$_kA_i \geq 2$.
\end{lemma}
\begin{proof}
Since $\Lambda$ is representation-finite, then, by Proposition \ref{prop:rep},  $\Gamma$ is a Dynkin quiver. Now, since $Q$ has no loops, and since we are assuming that dim$_k A_i \geq 2$, we infer that $|(\Sigma_i)_0| \geq 2$. 
If $i$ is not an end point of $\Gamma$, then $\overline{\Gamma}$ will contain a subgraph of type $\A_3$ with $i$ as the middle vertex.
So $Q$ will have a subquiver with only arrows of type I whose underlying graph is 
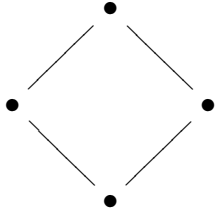
\begin{displaymath}
\xymatrix{ &  \bullet \ar@{-}[dr] & & \\ \bullet \ar@{-}[ur] \ar@{-}[dr] &  & \bullet \\ & \bullet \ar@{-}[ur] & & }
\end{displaymath}
a contradiction to $\Lambda$ being representation-finite. \\
Now, if $\Gamma$ is of type $\D_n, \E_6, \E_7, \E_8$ with $i$  an end point, then  $\oQ$ would have a subgraph of the type:
{\small 
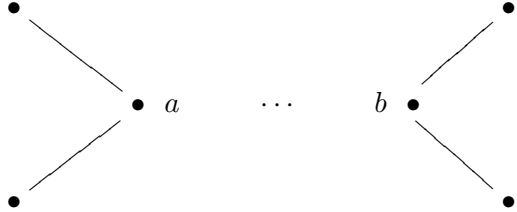
\begin{displaymath}
\xymatrix{\bullet \ar@{-}[dr]&&& &   \bullet \\
 &\ \ \ \bullet \ \ a& \cdots &  b\ \ 
 \ar@{-}[dr] \ar@{-}[ur] \bullet&  \\
\bullet \ar@{-}[ur]&&& &   \bullet }
\end{displaymath} }

where the points $a$ and $b $ could coincide or be linked by a sequence of edges  of type I. This also contradicts the fact that $\Lambda$ is representation-finite. It remains to show the last statement. For that, suppose $\Gamma$ is of type $\A_n, n\geq 2$ and suppose dim$_k A_1 \geq 2$ and dim$_k A_n \geq 2$. If $n=2$, then $\oQ$ will have a subpath of the kind
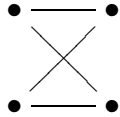
\begin{displaymath}
\xymatrix{\bullet \ar@{-}[r] \ar@{-}[dr] & \bullet \\
\bullet \ar@{-}[r] \ar@{-}[ur] & \bullet}
\end{displaymath}
with only arrows of type I and, if $n > 2$, then $\oQ$ will have a subpath of the kind
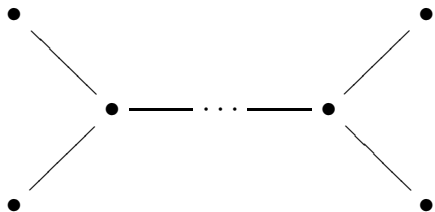
\begin{displaymath}
\xymatrix{\bullet \ar@{-}[dr]&&& &   \bullet \\
 &\bullet \ar@{-}[r]& \cdots \ar@{-}[r] & 
 \ar@{-}[dr] \ar@{-}[ur] \bullet&  \\
\bullet \ar@{-}[ur]&&& &   \bullet }
\end{displaymath}
with only arrows of type I. In both cases, this would imply that $\Lambda$ is not representation-finite, a contradiction to the hypothesis. This finishes the proof. 
\end{proof}

We shall now prove Theorem \ref{repfinite}.

\begin{proof} (Theorem \ref{repfinite}). 
Assume $\Lambda$ is representation-finite. By Proposition~\ref{prop:rep} and Lemma \ref{lemrepfinite}, $\Gamma$ is of Dynkin type and, if $\Gamma$ is of type $\D_n$ ($n \geq 4$),  $\E_6, \E_7$ or $ \E_8$, then the algebras $A_i$ would be equal to $k$ for any vertex $i \in \Gamma_0$, which is the same to say that $\Lambda$ is indeed a path algebra. This proves (i) and the last statement for these cases. \\
It remains to consider the case which $\Gamma$ is of type $\A_n$, with $n \geq 2$. By Lemma \ref{lemrepfinite}, at most one of the two end points, $1$ or $n$, of $\Gamma$ has its assigned algebra with dimension greater or equal to 2. Without loss of generality, assume $A_j = k$ for $j = 1, \cdots, n-1$ and that dim$_k A_n = m$ and write $A_n = k\Sigma_n/I_n$. Suppose $\Sigma_n$ has an arrow $\alpha$. Since $Q$ has no loops, we infer that $|(\Sigma_n)_0| \geq 2$ and then $\overline{Q}$ has a subgraph 
{\small 
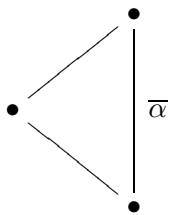
\begin{displaymath}
\xymatrix{& \bullet \ar@{-}[dd]\\
\bullet \ar@{-}[ur] \ar@{-}[dr] & \ \ \ \ \ \overline{\alpha}\\
& \bullet}
\end{displaymath} }
where all edges except $\overline{\alpha}$ are of type I (and so, by Remark \ref{typeI}, without relations in the corresponding subquiver). This contradicts the fact that $\Lambda$ is representation-finite. So, $\Sigma_n$ has no arrows and $A_n = k^m$. \\
If now $m \geq 4$, then $\overline{Q}$ will have a subgraph 

{\small 
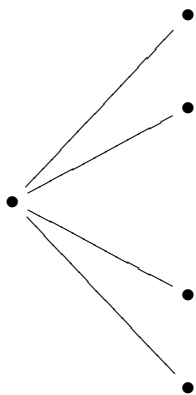
\begin{displaymath}
	\xymatrix{ && \bullet \\
	&& \bullet \\
	\bullet \ar@{-}[urr] \ar@{-}[uurr] \ar@{-}[drr] \ar@{-}[ddrr] && \\
	&& \bullet \\
	&& \bullet}
\end{displaymath}	}

with all arrows of type I, also a contradiction to the representation-finiteness of $\Lambda$. If $n =2$, this gives (ii). If now $n \geq 3$, then $m \leq 2$, otherwise, $\oQ$ will also have a subgraph 
{\small 
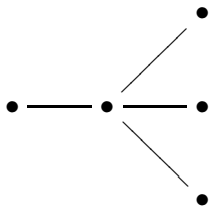
\begin{displaymath}
\xymatrix{ & &  \bullet \\
\bullet \ar@{-}[r] &\bullet \ar@{-}[r] \ar@{-}[ur]\ar@{-}[dr] & \bullet  \\
& &  \bullet }
\end{displaymath}  }
with all arrows of type I, and again, $\Lambda$ would not be representation-finite. This shows (iii). Last statement is now easy to check. \\
Conversely, assume that (i) (ii) or (iii) holds. Then $Q$ is Dynkin and so $\Lambda$ is representation-finite because of Theorem \ref{dynkin}, completing the proof of the theorem.
\end{proof}


\section{Tameness}\label{tameness}

Now, we concentrate on algebras of tame type. We shall not give a formal definition of tameness,  referring to \cite{Barot,Kir} for it. We shall however base our considerations on the theorem below which characterizes the path algebras of tame type (see \cite{Kir} for a proof) and prove a result in the same fashion of the one in Section 3. We shall use the expression {\bf strict tame} for a tame algebra which is not representation-finite. 

\vspace{.3 cm}

\begin{theorem} \label{euclidean}
Let $A = kQ$ be a path algebra given by $Q$.
Then $A$ is of strict tame type if and only if the underlying graph of $Q$ is one of the following\\
{\small 
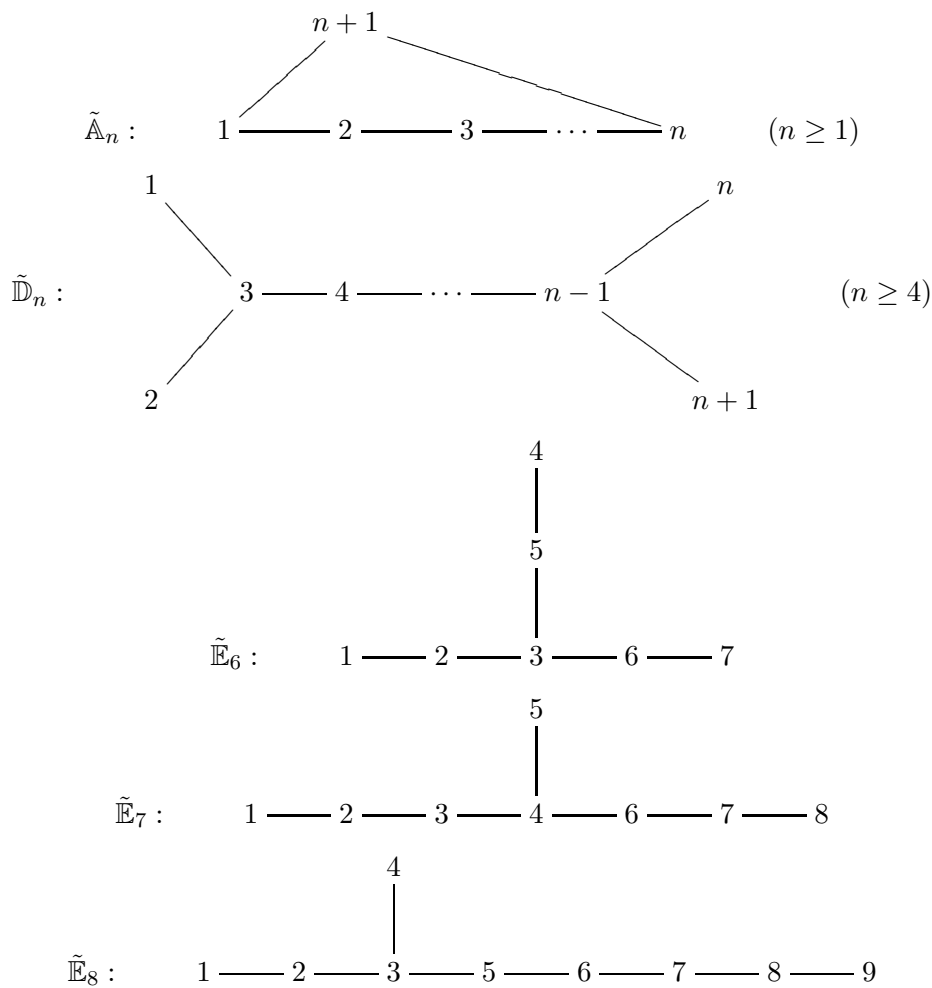
\begin{displaymath}
\xymatrix{&& n+1 \ar@{-}[drrr]&&&& \\
\tilde{\mathbb{A}}_n: & 1 \ar@{-}[ur] \ar@{-}[r] & 2 \ar@{-}[r] & 3 \ar@{-}[r] & \cdots \ar@{-}[r] & n & (n\geq 1)}
\end{displaymath}
\begin{displaymath}
\xymatrix{& 1\ar@{-}[dr]&&&&&n& \\
\tilde{\mathbb{D}}_n: &  & 3 \ar@{-}[r] & 4 \ar@{-}[r] & \cdots \ar@{-}[r] & n-1 \ar@{-}[ur] \ar@{-}[dr]&& (n\geq 4)\\
& 2 \ar@{-}[ur]&&&&&n+1&}
\end{displaymath}
\begin{displaymath}
\xymatrix{& & &4 \ar@{-}[d]&& \\
& & &5 \ar@{-}[d]&& \\
\tilde{\mathbb{E}}_6: & 1 \ar@{-}[r]& 2 \ar@{-}[r] & 3 \ar@{-}[r] & 6 \ar@{-}[r] & 7}
\end{displaymath}
\begin{displaymath}
\xymatrix{& & &&5 \ar@{-}[d]&& \\
\tilde{\mathbb{E}}_7: & 1 \ar@{-}[r]& 2 \ar@{-}[r] & 3 \ar@{-}[r] & 4 \ar@{-}[r] & 6 \ar@{-}[r] & 7 \ar@{-}[r]& 8}
\end{displaymath}
\begin{displaymath}
\xymatrix{& & &4 \ar@{-}[d]&& &&\\
\tilde{\mathbb{E}}_8: & 1 \ar@{-}[r]& 2 \ar@{-}[r] & 3 \ar@{-}[r] & 5 \ar@{-}[r] & 6 \ar@{-}[r] & 7 \ar@{-}[r]& 8 \ar@{-}[r]& 9}
\end{displaymath} }
\end{theorem}
\vspace{.3 cm}

It should be observed  that since we are assuming that the algebra $A$ is finite dimensional, then it is not possible to have oriented cycles in the quiver $Q$ in case the underlying graph is ${\tilde{\A}}_n$. All other possibilities of orientations are permitted. The graphs listed in the above theorem are called {\bf Euclidean} and, as before, we shall keep their numbering for the sequel. 


Theorem~\ref{euclidean} will be our main tool for deciding whether or not an algebra is tame. However, we will later run across an exceptional family of algebras whose representation-type will require additional, more involved techniques to analyse.Therefore what we do next is to deal with this exceptional case.

Let $B_I$ be the algebra given by the quiver 

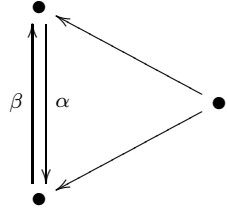
\begin{displaymath}
\xymatrix{\bullet \ar@<.5ex>[dd]^{\alpha}& & \\
&& \bullet \ar[ull] \ar[dll] \\
\bullet  \ar@<.5ex>[uu]^{\beta}&&}
\end{displaymath}
bound by an admissible  ideal $I$ which is generated by paths involving only the arrows $\alpha$ and $\beta$.

\begin{lemma} \label{lem:exception}
The algebras of the kind $B_I$ and $B_I^{op}$ are not representation-finite and they are tame if and only if $I = < \alpha \beta, \beta \alpha >$.
\end{lemma}

\begin{proof}
We shall prove the statement for algebras of the type $B_I$, the other case being similar. 
Denote the vertices of the ordinary quiver $Q_I$ of $B_I$ as follows:
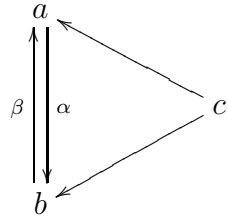
\begin{displaymath}
\xymatrix{a \ar@<.5ex>[dd]^{\alpha}& & \\
&& c \ar[ull] \ar[dll] \\
b \ar@<.5ex>[uu]^{\beta}&&}
\end{displaymath}

Since $Q_I$ has a subquiver of type $\tilde{\mathbb{A}}_2$ with no relations involved, we infer that $B_I$ is not representation-finite. If now $I = <\alpha \beta, \beta \alpha>$, then $B_I$ is a \textit{string} algebra, and so it must be tame by \cite{BR}. It remains to show that if $B_I$ is tame, then $\alpha \beta$ and $\beta \alpha$ both belong to $I$. Suppose this does not happen. Now $(Q,I)$ admits a covering

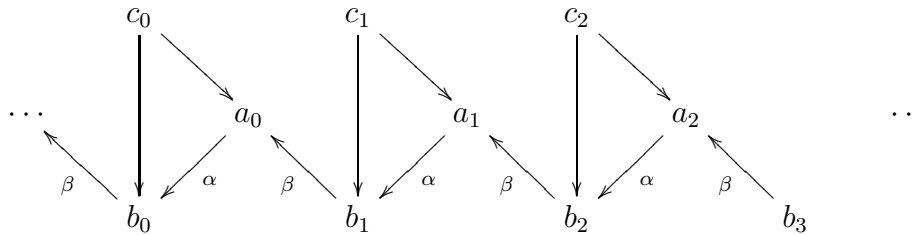
\begin{displaymath}
\xymatrix{
& c_0 \ar[dd] \ar[dr] && c_1 \ar[dd] \ar[dr]&& c_2 \ar[dd] \ar[dr]&&& \\
 \cdots && a_0 \ar[dl]^{\alpha} && a_1 \ar[dl]^{\alpha}&& a_2 \ar[dl]^{\alpha}&& \cdots \\
& b_0 \ar[ul]^{\beta} && b_1 \ar[ul]^{\beta}&& b_2 \ar[ul]^{\beta}&& b_3 \ar[ul]^{\beta} &}
\end{displaymath}

If $\alpha \beta \notin I$, then this covering has the following subquiver with no relations:

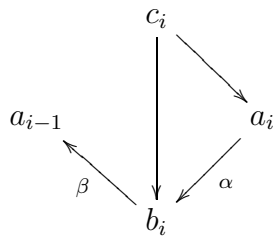
\begin{displaymath}
\xymatrix{& c_i \ar[dd] \ar[dr] & \\
a_{i-1} && a_i \ar[dl]^{\alpha}\\
& b_i \ar[ul]^{\beta}& }
\end{displaymath}

Similarly, if $\beta \alpha \notin I$, then it has the following subquiver without relations:

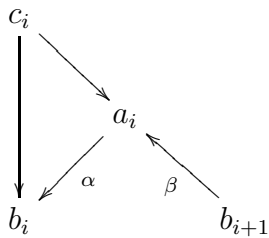
\begin{displaymath}
\xymatrix{c_i \ar[dd] \ar[dr] && \\
& a_i \ar[dl]^{\alpha} & \\
b_i && b_{i+1} \ar[ul]^{\beta}}
\end{displaymath}

In both cases the subquivers are non-Euclidean, which implies that the category of (finitely-generated) representations over the covering above is wild. Therefore, $B_I$ must be itself wild, see for instance \cite{DS} or \cite{dlP}.

\end{proof}


The main result on tameness is the following.

\begin{theorem} \label{reptame}
Let $\Lambda = k(\Gamma, \Aa) \cong kQ/I$ be a gp-algebra and its realization as a quotient of a path algebra. Assume that none of the ordinary quivers $\Sigma_i$ of the algebras $A_i$ in $\Aa$ have loops.  Then $\Lambda$ is of strict tame type if and only if the underlying graph of $\Gamma$ is Euclidean or of type $\A_n$ ($n \geq 1$) or $\D_n$ ($n\geq 4$) and 
\benu
\item[(i)]  If $\Gamma$ is Euclidean, then $A_i = k$ for every $i \in \Gamma_0$.
\item[(ii)] If $\Gamma$ is $ \D_4$, then one and only one of the algebras $A_1, A_2$ or $ A_4$ equals to  $k^2$ and the other algebras of $\Aa$ equal  $k$. 
\item[(iii)] If $\Gamma$ is $ \D_n$ with $n \geq 5$, then $A_n = k^2$ and   $A_j = k$ if $j \neq n$.
\item[(iv)] If $\Gamma$ is $ \A_1$, then  $A_1$ is of strict tame type.
\item[(v)]   If  $\Gamma$ is $\A_2$, then either
\begin{enumerate}
\item $A_1 = A_2 = k^2$; or
\item one of $A_1$ and $A_2$ is $k$ and the other is one of the following possibilities: $k^4$, $k\A_2$, or the radical square zero algebra given by the quiver 
\begin{displaymath}
 \xymatrix{a \ar@<.5ex>[rr]^{\alpha} && b \ar@<.5ex>[ll]^{\beta}}
\end{displaymath}
\end{enumerate}
\item[(vi)] If $\Gamma$ is $\A_3$, then either \begin{enumerate}
\item $A_1 = A_3 = k$ and $A_2 = k^2$; or
\item $A_1 = A_3 = k^2$ and $A_2 = k$; or
\item one and only one of $A_1$ or $A_3$ is $k^3$ and the other algebras  equal to $k$.
\end{enumerate}
\item[(vii)] If $\Gamma$ is  $\A_n$ with $n \geq 4$, then  $A_1 = A_n = k^2$ and $A_i = k$ for every $i \neq 1,n$.
\enu
\end{theorem}
\vspace{.3 cm}

As before, we start with a lemma. 
\vspace{.3 cm}

\begin{lemma}  \label{lemtame}
Let $\Lambda = k(\Gamma, \Aa) \cong kQ/I$ be a  gp-algebra of tame type such that none of the ordinary quivers $\Sigma_i$ of the algebras $A_i$ in $\Aa$ has loops. Suppose  dim$_kA_i \geq 2$ for some $i \in \Gamma_0$.
\benu
\item[(i)]  If $i$ is not an end point of $\Gamma$, then $\Gamma$ is of type $\A_3$. Also, in this case, $A_2=k^2$ and $A_1= A_3 = k$; 
\item[(ii)]  If $i$ is an end point of $\Gamma$, then  $\Gamma$ is either of type $\A_n$ for some $n \geq 1$ or of type $\D_n$, with $n\geq 4$.
\enu
\end{lemma}
\begin{proof}
Since $\Lambda$ is of tame type, then $\Gamma$ is either Dynkin or Euclidean because of Proposition~\ref{prop:rep} and Theorems   \ref{dynkin}  and \ref{euclidean}. \\
(i) Assume that $i$ is not an end point of $\Gamma$. In particular, the cases $\A_1$ and $\A_2$ are not possible.  If now $\Gamma$ is of type $\A_3$, then, since $i$ is not an end point, $i = 2$ (in the numbering we have fixed above). In this case, dim$_k A_2 = 2$ because otherwise $\oQ$ will have a subgraph 
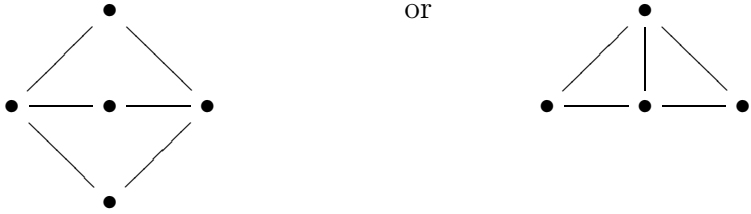
\begin{displaymath}
\xymatrix{ &  \bullet \ar@{-}[dr] & &  \\ \bullet \ar@{-}[ur] \ar@{-}[dr] \ar@{-}[r] & \bullet \ar@{-}[r] &  \bullet \\ & \bullet \ar@{-}[ur] & &  }
\hspace{1.3cm}
\text{or}
\hspace{1.3cm}
\xymatrix{ &  \bullet \ar@{-}[dr] \ar@{-}[d] & &  \\ \bullet \ar@{-}[ur]  \ar@{-}[r] & \bullet \ar@{-}[r] &  \bullet  }
\end{displaymath}
with at most one edge of type II, again a contradiction (see Remark  \ref{typeI}). Also, in this case, $A_1 = A_3 = k$, otherwise, $\oQ$ will have a subgraph of type 
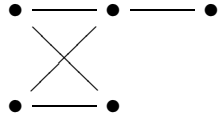
\begin{displaymath}
\xymatrix{\bullet \ar@{-}[r] \ar@{-}[dr] & \bullet \ar@{-}[r] & \bullet \\
\bullet \ar@{-}[r] \ar@{-}[ur] & \bullet &  }
\end{displaymath}
with at most one edge of type II, leading to a contradiction. This gives the only possibility in item (i).   \\
It remains to analyse the cases where $\Gamma$ is of types $\A_n$ ($n \geq 4$), $\D_n$ ($n\geq 4$), $\E_6$, $\E_7$, $\E_8$, $\tilde{\mathbb{A}}_n$ ($n \geq 1$), $\tilde{\mathbb{D}}_n$ ($n\geq 4$), $\tilde{\mathbb{E}}_6$, $\tilde{\mathbb{E}}_7$ and $\tilde{\mathbb{E}}_8$ and show that they are not possible. If $\Gamma$ is of type $\tilde{\mathbb{A}}_1$ (that is, $\Gamma$ is the Kronecker quiver), then $\oQ$ will have a subgraph of type 
{\small 
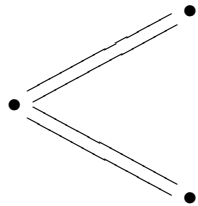
\begin{displaymath}
 \xymatrix{ && \bullet\\
 \bullet \ar@<.5ex>@{-}[urr] \ar@<-.5ex>@{-}[urr] \ar@<.5ex>@{-}[drr] \ar@<-.5ex>@{-}[drr] && \\
 && \bullet }
\end{displaymath}	}
with only edges of type I ans so $\Lambda$ would not be of tame type, a contradiction. If $\Gamma$ is of type $\tilde{\mathbb{A}}_2$, then $\oQ$ will have a subgraph of type

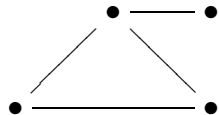
\begin{displaymath}
\xymatrix{& \bullet \ar@{-}[r] \ar@{-}[dr]& \bullet \\    
		\bullet \ar@{-}[rr] \ar@{-}[ur]& & \bullet}
\end{displaymath}

with only edges of type I, also contradicting the fact that $\Lambda$ is of tame type. \\
 Observe that in all the others  cases, $\overline{\Gamma}$ will have either a subgraph  $\bullet -- \bullet -- i -- \bullet$ or a subgraph $\bullet -- i -- \bullet -- \bullet$. In both cases, $\oQ$ will contain a subgraph 
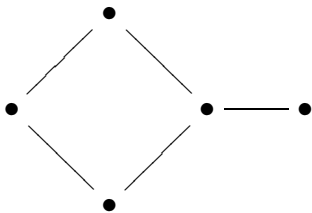
\begin{displaymath}
\xymatrix{ &  \bullet \ar@{-}[dr] & & & \\ \bullet \ar@{-}[ur] \ar@{-}[dr] &  & \bullet \ar@{-}[r] & \bullet \\ & \bullet \ar@{-}[ur] & & &  }
\end{displaymath}
with all edges of type I, a contradiction to $\Lambda$ being tame. This proves (i). \\
(ii) Assume now that $i$ is an end point of $\Gamma$ and suppose $\Gamma$ is neither of type $\A_n$ nor of type $\D_n$ ($n\geq 4)$. Observe also that $\Gamma$ is not of type $\tilde{\mathbb{A}}_n$ because  there are no end points in such quivers. Then $\oQ$ will have a subgraph of one of the following types:
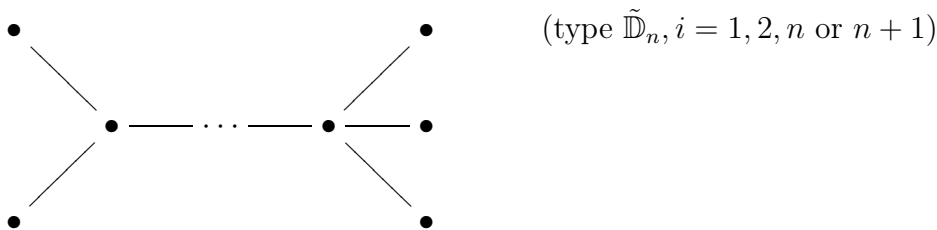
\begin{displaymath}
\xymatrix{\bullet \ar@{-}[dr]&&& &   \bullet \\
 &\bullet \ar@{-}[r]& \cdots\ar@{-}[r] & \bullet \ar@{-}[r]
 \ar@{-}[dr] \ar@{-}[ur] & \bullet \\
\bullet \ar@{-}[ur]&&& &   \bullet }
\hspace{1.3cm}
(\mbox{type } \tilde{\mathbb{D}}_n,  i = 1, 2, n\mbox{ or }n+1)
\end{displaymath}
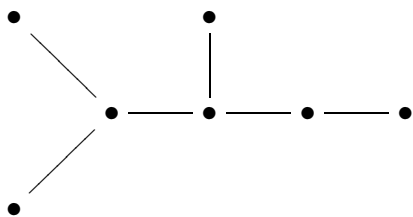
\begin{displaymath}
\xymatrix{\bullet \ar@{-}[dr]&& \bullet \ar@{-}[d]& &   \\
 &\bullet \ar@{-}[r]& \bullet \ar@{-}[r] & \bullet
 \ar@{-}[r]  & \bullet\\
\bullet \ar@{-}[ur]&&& &   }
\hspace{1.3cm}
\begin{array}{l} (\mbox{type } {\mathbb{E}}_6,  i = 1, 6)\\
(\mbox{types } {\mathbb{E}}_7, {\mathbb{E}}_8,  \tilde{\mathbb{E}}_8, i = 1)\\
(\mbox{types } \tilde{\mathbb{E}}_6,  i = 1,4,7)
\end{array}
\end{displaymath}
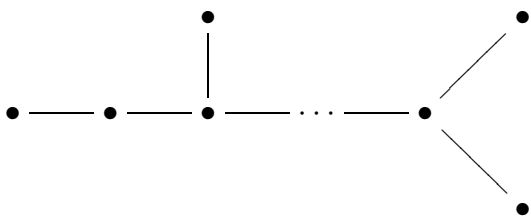
\begin{displaymath}
\xymatrix{& & \bullet \ar@{-}[d] & & & \bullet \\
\bullet \ar@{-}[r] & \bullet \ar@{-}[r] &\bullet \ar@{-}[r]& \cdots\ar@{-}[r] & \bullet 
 \ar@{-}[dr] \ar@{-}[ur] & \\ &&&&& \bullet }
\hspace{1.3cm}
\begin{array}{l} (\mbox{type } {\mathbb{E}}_7,  i = 7)\\
(\mbox{type } {\mathbb{E}}_8,   i = 8)\\
(\mbox{type } \tilde{\mathbb{E}}_7,  i = 1, 8) \\
(\mbox{type } \tilde{\mathbb{E}}_8,  i = 9)
\end{array}
\end{displaymath}
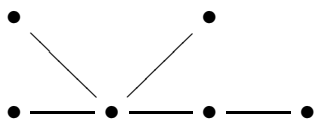
\begin{displaymath}
\xymatrix{\bullet \ar@{-}[dr] & & \bullet & \\
\bullet \ar@{-}[r] & \bullet \ar@{-}[r] \ar@{-}[ur]&\bullet \ar@{-}[r] & \bullet }
\hspace{1.3cm}
\begin{array}{l} (\mbox{types } {\mathbb{E}}_6, {\mathbb{E}}_7, {\mathbb{E}}_8,  i = 4)\\
(\mbox{type } \tilde{\mathbb{E}}_7,  i = 5) \\
(\mbox{type } \tilde{\mathbb{E}}_8,  i = 4)
\end{array}
\end{displaymath}
involving only edges of type I, contradicting the fact that $\Lambda$ is tame. This proves (ii). 
\end{proof}

We shall now prove Theorem \ref{reptame}.

\begin{proof} (Theorem \ref{reptame}). 
Since $\Lambda$ is of tame type, then $Q$ cannot have a proper subquiver without relations of Euclidean type. So, $\Gamma$ is either Dynkin or Euclidean.\\ 
Assume there is an $i \in \Gamma_0$ such that dim$_kA_i   \geq 2$. By Lemma \ref{lemtame}, $\Gamma$ is either of type $\A_n$ ($n\geq 1)$ or of type $\D_n$ ($n\geq 4)$. \\ 
So, if $\Gamma$ is Euclidean or Dynkin of type $\E_p$ ($p =6,7,8$), then $A_j = k$, for each $j \in \Gamma_0$,  $Q = \Gamma$ and $\Lambda$ is the path algebra $kQ$. Because we are assuming $\Lambda$ to be of strict tame type, the Dynkin graphs above are excluded and (i) is proved. \\
It remains to analyse the cases listed in Lemma \ref{lemtame}, they are $\A_n $ ($n \geq 1$) or $\D_n$ $(n \geq 4)$ upon the existence of the vertex $i$ which satisfies dim$_kA_i = m  \geq 2$. Whe shall keep this hypothesis on the vertex $i$ for the rest of our proof. Also, as befores, se shall write for each $j\in(\Gamma_0$, $A_j = k\Sigma_j/I_j$. \\
Consider first that $\Gamma$ is of type $\D_n$. Because of Lemma \ref{lemtame}, the vertex $i$ has to be an end point. We show now that there exists only one of such vertex and $m = 2$. If $\Gamma$ is $\D_4$ and two of $A_1, A_2$ or $A_4$ have dimension greater than one or if $m \geq 3$, then $\oQ$ will contain a  subgraph
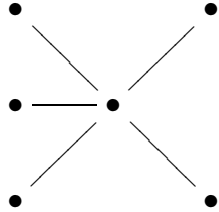
\begin{displaymath}
\xymatrix{\bullet \ar@{-}[dr] && \bullet\\
\bullet\ar@{-}[r] & \bullet \ar@{-}[dr] \ar@{-}[ur]& \\
\bullet \ar@{-}[ur] && \bullet\\
}
\end{displaymath}
with only edges of type I, which is not of tame type. Since $Q$ is not Dynkin, because it is not representation-finite,  the only possibility is that one and only one of the $A_1, A_2$ or $A_4$ is $k^2$ and this proves (ii). \\
(iii)  Now, assume that $\Gamma = \D_n$ with $n\geq 5$.  If dim$_kA_i \geq 2$ for either $i =1$ or $i = 2$, then $\Sigma_i$ will have at least 2 vertices (because it has no loops) and $\oQ$ will have a subgraph
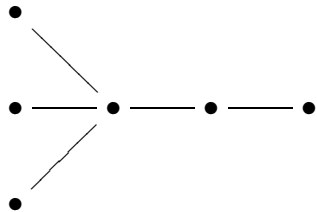
\begin{displaymath}
\xymatrix{\bullet \ar@{-}[dr]&&& \\
\bullet \ar@{-}[r] &\bullet \ar@{-}[r]&\bullet \ar@{-}[r]&\bullet \\
\bullet \ar@{-}[ur]&&&}
\end{displaymath}
with only edges of type I, which is not of tame type. So $A_1 = A_2 = k$ and, since $i$ is an end point, $i = n$ and $A_j = k$ for $j=3, \cdots, n-1$. 
If dim$_k A_n \geq 3$, then $\oQ$ will contain a subgraph as follows
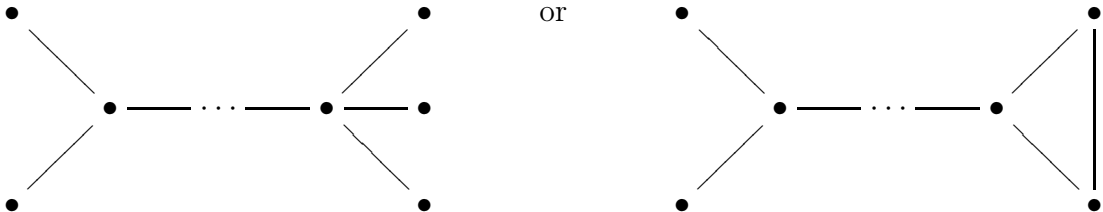
\begin{displaymath}
\xymatrix{\bullet \ar@{-}[dr]&&& &   \bullet \\
 &\bullet \ar@{-}[r]& \cdots\ar@{-}[r] & \bullet \ar@{-}[r]
 \ar@{-}[dr] \ar@{-}[ur] & \bullet \\
\bullet \ar@{-}[ur]&&& &   \bullet }
\hspace{1.3cm}
\text{or}
\hspace{1.3cm}
\xymatrix{\bullet \ar@{-}[dr]&&& &   \bullet \ar@{-}[dd]\\
 &\bullet \ar@{-}[r]& \cdots\ar@{-}[r] & \bullet
 \ar@{-}[dr] \ar@{-}[ur] & \\
\bullet \ar@{-}[ur]&&& &   \bullet }
\end{displaymath}
with at most one arrow of type II, none of them of tame type. So, dim$_k A_n = 2$ and $A_n = k^2$ because $Q$ has no loops. This proves (iii).\\
Let us analyse now the case $\Gamma$ of type $\A_n$. \\
(iv) If $\Gamma$ is $\A_1$,  then $\Lambda \cong A_1$ and so, $A_1$ is of strict tame type as required. \\
(v) Suppose $\Gamma$ is $\A_2$, let us say, $1 \lra 2$. In case dim$_kA_1$ and dim$_k A_2$ are both greater or equal to two, then $A_1 = A_2 = k^2$ since otherwise $\oQ$ will have one of the following  subgraphs
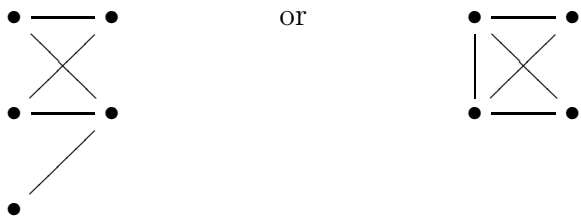
\begin{displaymath}
\xymatrix{\bullet \ar@{-}[r] \ar@{-}[dr] & \bullet \\
\bullet \ar@{-}[r] \ar@{-}[ur] & \bullet \\
\bullet\ar@{-}[ur] & }
\hspace{2cm}
\text{or}
\hspace{2cm}
\xymatrix{\bullet \ar@{-}[r] \ar@{-}[dr] & \bullet \\
\bullet \ar@{-}[r] \ar@{-}[ur]\ar@{-}[u] & \bullet}
\end{displaymath}
with at most one edge of type II, leading to a contradiction. So, if dim$_kA_1 \geq 2 $ and dim$_k A_2 \geq 2$, then $A_1 = A_2 = k^2$, which is (v)(a). \\
Suppose without loss of generality that  dim$_k A_1 =m \geq 2$ and dim$_k A_2 = 1$. Write, as before, that $A_1 = k \Sigma_1/ I_1$. Let us analyse first the case where the quiver $\Sigma_1$ has no arrows. If $m =2$ or $3$, then $\oQ$ will be either $\A_3$ or $\D_4$, in both cases, representation-finite. On the other hand, if $m \geq 5$, then $\oQ$ will contain a subgraph of type 
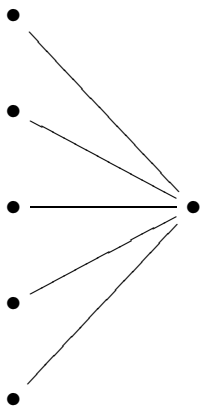
\begin{displaymath}
\xymatrix{ \bullet \ar@{-}[ddrr]&& \\
\bullet \ar@{-}[drr]&& \\
\bullet \ar@{-}[rr]&& \bullet \\
\bullet \ar@{-}[urr]&& \\
\bullet \ar@{-}[uurr]&&}
\end{displaymath}
with no edges of type II, a contradiction to tameness. So, $A_1= k^4$ which is one of the possibilities listed in (v)(b).\\
Suppose now that $\Sigma_1$  has at least one arrow. Because, by hypothesis, $\Sigma_1$ has no loops, the number of vertices has to be exactly two, because otherwise $\oQ $ would contain a subgraph 
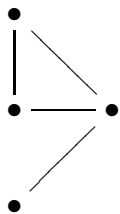
\begin{displaymath}
\xymatrix{\bullet  \ar@{-}[dr] &  \\
\bullet \ar@{-}[r] \ar@{-}[u] & \bullet \\
\bullet\ar@{-}[ur] & }
\end{displaymath}
with at most one edge of type II, which is not of tame type. If $\Sigma_1$ has only one arrow, then it is $\A_2$ and $A_1= k\A_2$, which is one of the possibilities listed in (v)(b) (and $\Gamma$ will be, in this case, of type $\tilde{\mathbb{A}}_2$) . \\
Assume then that $\Sigma$ has more than one arrow and let $a, b$ be its vertices. Observe that there are no two arrows going in the same direction because otherwise, $\Gamma$ will have one of the following subquivers
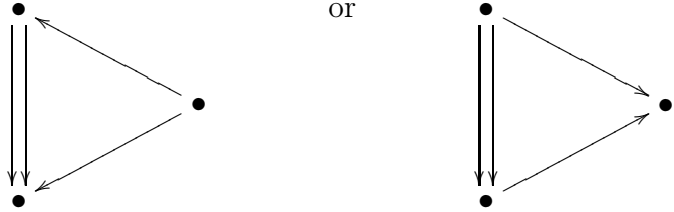
\begin{displaymath}
\xymatrix{\bullet \ar@<-.5ex>[dd] \ar@<.5ex>[dd]& & \\
&& \bullet \ar[ull] \ar[dll] \\
\bullet&&}
\hspace{ 1.5 cm} {\mbox{or}} \hspace{ 1.5 cm}
\xymatrix{\bullet \ar@<-.5ex>[dd] \ar@<.5ex>[dd] \ar[drr]& & \\
&& \bullet  \\
\bullet  \ar[urr]&&}
\end{displaymath}
with no relations, neither of them of tame type, a contradiction. Therefore, the only possibility left for $\Sigma_1$ is the quiver 
\begin{displaymath}
 \xymatrix{a \ar@<.5ex>[rr]^{\alpha} && b \ar@<.5ex>[ll]^{\beta}}
\end{displaymath}
and so, $\Lambda$ is one of  the algebras considered in Lemma \ref{lem:exception} which is tame if and only if $I_1 = < \alpha \beta, \beta\alpha >$. This finishes the proof of (v).  \\
(vi) Suppose now $\Gamma$ is $\A_3$ and, as before, $i$ is a vertex such that dim$_k A_i \geq 2$. If $i$ is not an end point of $\Gamma$, that is, $i = 2$, then, because of Lemma \ref{lemtame}, $A_2 = k^2$ and $A_1 = A_3 = k$ which is the possibility (vi)(a). \\
Suppose that $i$ is an end point of $\Gamma$, let us say $i = 1$. Also, because of Lemma \ref{lemtame}, $A_2 = k$. If also dim$_k A_3 \geq 2$, then $A_1 = A_3 = k^2$ since otherwise $\oQ$ will have a subgraph of type
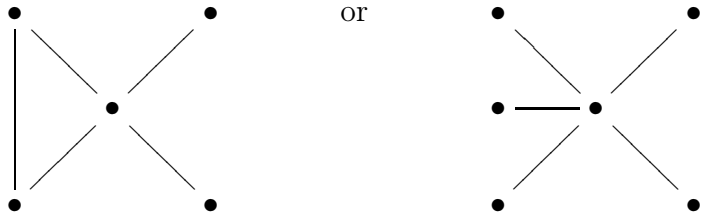
\begin{displaymath}
\xymatrix{\bullet \ar@{-}[dr]  \ar@{-}[dd]&& \bullet \\
 &\bullet \ar@{-}[ur] \ar@{-}[dr] & \\
\bullet \ar@{-}[ur]&& \bullet }
\hspace{ 1.5 cm} {\mbox{or}} \hspace{ 1.5 cm}
\xymatrix{\bullet \ar@{-}[dr] && \bullet\\
\bullet\ar@{-}[r] & \bullet \ar@{-}[dr] \ar@{-}[ur]& \\
\bullet \ar@{-}[ur] && \bullet\\
}
\end{displaymath}
with at most one arrow of type II, which are not tame. This gives case (vi)(b). It remains the case where dim$_k A_1 =  m \geq 2$, $A_2 = A_3 = k$ (and its dual $A_1 = A_2 = k$ and dim$_k A_1 = m \geq 2$ which is proven similarly). If $m = 2$, then $Q$ is of type $\D_4$ which is representation-finite. So $m \geq 3$. If $m \geq 4$  or if $\Sigma_1$ has an arrow, then $\oQ$ will have a subgraph of one of the types
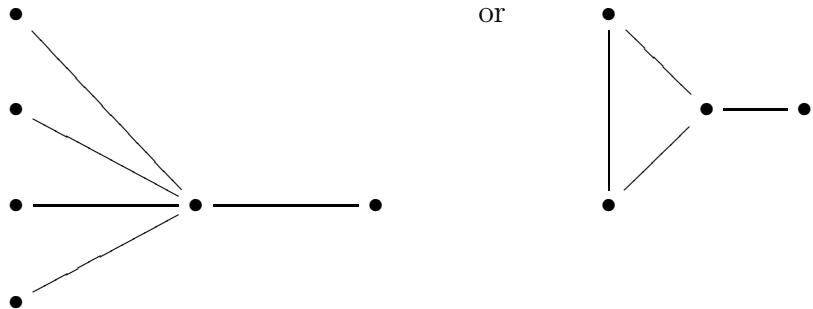
\begin{displaymath}
\xymatrix{ \bullet \ar@{-}[ddrr]&&&& \\
\bullet \ar@{-}[drr]&&&& \\
\bullet \ar@{-}[rr]&& \bullet \ar@{-}[rr] && \bullet \\
\bullet \ar@{-}[urr]&&&& }
\hspace{1 cm} { \mbox{ or }} \hspace{1 cm}
\vspace{.5 cm}\\ 
\xymatrix{\bullet \ar@{-}[dr]  \ar@{-}[dd]&& \\
 &\bullet \ar@{-}[r]&\bullet \\
\bullet \ar@{-}[ur]&&}
\end{displaymath}
with at most one arrow of type II, a contradiction. \\
So $m = 3$ and $\Sigma_1$ has no arrows, that is, $A_1 = k^3$ which is case (vi)(c) (and $\Gamma$ is $\tilde\D_4$).\\
(vii) Finally, assume $\Gamma$ is of type $\A_n$, with $n \geq 4$. Then, because of Lemma \ref{lemtame}, $i$ is an end point, let us say $i = 1$, that is,  dim$_k A_1 = m \geq 2$. If $m =2$, that is, $A_1 = k^2$ (because there are no loops in $\Sigma_1$), then $A_n = k^2$ because otherwise $\oQ$ is either $\D_n$ (which is not possible because the algebra is representation-infinite) of has a subgraph as follows
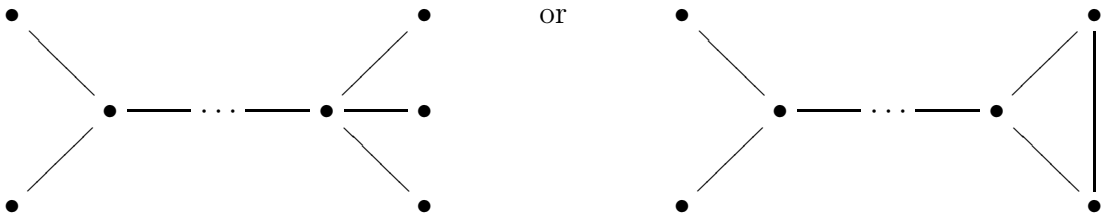
\begin{displaymath}
\xymatrix{\bullet \ar@{-}[dr]&&& &   \bullet \\
 &\bullet \ar@{-}[r]& \cdots\ar@{-}[r] & \bullet \ar@{-}[r]
 \ar@{-}[dr] \ar@{-}[ur] & \bullet \\
\bullet \ar@{-}[ur]&&& &   \bullet }
\hspace{1.3cm}
\text{or}
\hspace{1.3cm}
\xymatrix{\bullet \ar@{-}[dr]&&& &   \bullet \ar@{-}[dd]\\
 &\bullet \ar@{-}[r]& \cdots\ar@{-}[r] & \bullet
 \ar@{-}[dr] \ar@{-}[ur] & \\
\bullet \ar@{-}[ur]&&& &   \bullet }
\end{displaymath}
with at most one edge of type II, a contradiction. Finally, if $m \geq 3$, then $\oQ$ will  contain a subgraph as follows (recall that $Q$ has no loops):
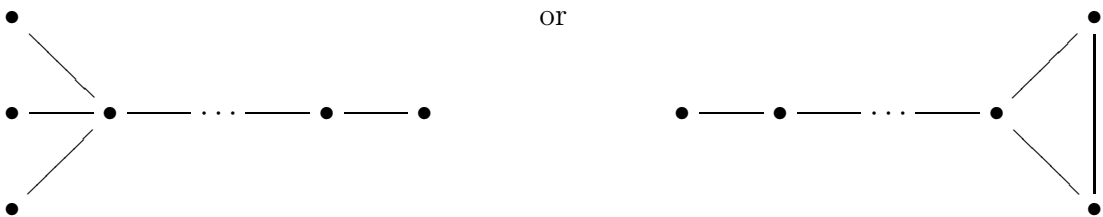
\begin{displaymath}
\xymatrix{\bullet \ar@{-}[dr]&&& &   \\
 \bullet \ar@{-}[r] &\bullet \ar@{-}[r]& \cdots\ar@{-}[r] & \bullet \ar@{-}[r]
  & \bullet \\
\bullet \ar@{-}[ur]&&& &   }
\hspace{1.3cm}
\text{or}
\hspace{1.3cm}
\xymatrix{ &&& &   \bullet \ar@{-}[dd]\\
 \bullet\ar@{-}[r] &\bullet \ar@{-}[r]& \cdots\ar@{-}[r] & \bullet
 \ar@{-}[dr] \ar@{-}[ur] & \\
 &&& &   \bullet }
\end{displaymath}
with at most one edge of type II, none of them of tame type. So dim$_kA_1 =$ dim$_k A_n = 2$, giving (vii). \\
Now, for the converse, it is not difficult to see that all the algebras listed in (i) to (vii) are strict tame. 
This finishes the proof. 
\end{proof}

\section{Acknowledgments}

The authors gratefully acknowledge financial support by S\~ao Paulo Research Foundation (FAPESP), (grant numbers \#2020/13925-6 and \#2022/02403-4) and by CNPq (Pq 312590/2020-2).

\end{document}